\documentclass[12pt]{amsart}

\usepackage{amscd}
\usepackage{amsmath}
\usepackage{amssymb}
\usepackage{mathrsfs} 			

\usepackage[all]{xy}

\def\frk{\frak}               
\def\aa{{\frk a}}
\def\pp{{\frk p}}

\def\qq{{\frk q}}

\def\mm{{\frk m}}

\def\Phi{{\frk n}}
\def\Phi{{\frk N}}
\def\bb{{\bold b}}

\def\opn#1#2{\def#1{\operatorname{#2}}} 
\opn\chara{char} \opn\length{\ell} \opn\pd{pd} \opn\rk{rk}
\opn\projdim{proj\,dim} \opn\injdim{inj\,dim} \opn\rank{rank}
\opn\depth{depth} \opn\sdepth{sdepth} \opn\fdepth{fdepth}
\opn\grade{grade} \opn\height{height} \opn\embdim{emb\,dim}
\opn\codim{codim}  \opn\min{min} \opn\max{max}

\opn\Tr{Tr} \opn\bigrank{big\,rank}
\opn\superheight{superheight}\opn\lcm{lcm}
\opn\trdeg{tr\,deg}
\opn\reg{reg} \opn\lreg{lreg} \opn\ini{in} \opn\lpd{lpd}
\opn\size{size}
\opn\div{div} \opn\Div{Div} \opn\cl{cl} \opn\Cl{Cl}
\opn\Spec{Spec} \opn\Supp{Supp} \opn\supp{supp} \opn\Sing{Sing}
\opn\Ass{Ass} \opn\Min{Min}
\opn\Ann{Ann} \opn\Rad{Rad} \opn\Soc{Soc}
\opn\Im{Im} \opn\Ker{Ker} \opn\Coker{Coker} \opn\Am{Am}
\opn\Hom{Hom} \opn\Tor{Tor} \opn\Ext{Ext} \opn\End{End}
\opn\Aut{Aut} \opn\id{id}  \opn\deg{deg}

\opn\Gin{Gin} \opn\Hilb{Hilb}
\opn\aff{aff} \opn\con{conv} \opn\relint{relint} \opn\st{st}
\opn\lk{lk} \opn\cn{cn} \opn\core{core} \opn\vol{vol}
\opn\link{link} \opn\star{star}
\opn\gr{gr}

\def\pot#1#2{#1[\kern-0.28ex[#2]\kern-0.28ex]}

\opn\dirlim{\underrightarrow{\lim}}
\opn\inivlim{\underleftarrow{\lim}}
\let\to=\rightarrow

\def\Implies{\ifmmode\Longrightarrow \else
        \unskip${}\Longrightarrow{}$\ignorespaces\fi}
\def\implies{\ifmmode\Rightarrow \else
        \unskip${}\Rightarrow{}$\ignorespaces\fi}
\def\iff{\ifmmode\Longleftrightarrow \else
        \unskip${}\Longleftrightarrow{}$\ignorespaces\fi}

\let\:=\colon
\newtheorem{Theorem}{Theorem}[]
\newtheorem{Lemma}[Theorem]{Lemma}
\newtheorem{Corollary}[Theorem]{Corollary}

\newtheorem{Remark}[Theorem]{Remark}

\newtheorem{Conjecture}[Theorem]{Conjecture}
\newtheorem{Question}[Theorem]{Question}

\let\epsilon\varepsilon
\let\phi=\varphi
\let\kappa=\varkappa
\def\qed{\ifhmode\textqed\fi
      \ifmmode\ifinner\quad\qedsymbol\else\dispqed\fi\fi}
\def\textqed{\unskip\nobreak\penalty50
       \hskip2em\hbox{}\nobreak\hfil\qedsymbol
       \parfillskip=0pt \finalhyphendemerits=0}
\def\dispqed{\rlap{\qquad\qedsymbol}}

\opn\dis{dis}
\def\pnt{{\raise0.5mm\hbox{\large\bf.}}}

\opn\Lex{Lex}

\begin{document}

\title{ The Bass-Quillen Conjecture and Swan's question}

\author{ Dorin Popescu }

\address{Dorin Popescu, Simion Stoilow Institute of Mathematics of the Romanian Academy, Research unit 5,
University of Bucharest, P.O.Box 1-764, Bucharest 014700, Romania}
\email{dorin.popescu@imar.ro}

\begin{abstract} We present a question which implies a complete  positive answer for the Bass-Quillen Conjecture.

\end{abstract}

\maketitle

\vskip 0.5 cm

\section*{Introduction}

The theory of projective modules over polynomial algebras over a regular ring $R$ was an important subject in Commutative Algebra starting with Serre's Conjecture (see \cite{La}) and its extension considered by Bass and Quillen.

\begin{Conjecture} (Bass-Quillen Conjecture, \cite[Problem IX]{B}, \cite{Q}) Let $R$ be a regular ring. Every finitely generated projective module $P$ over a polynomial $R$-algebra, $R[T]$, $T=(T_1,\ldots, T_n)$ is extended from $R$, i.e. 
$P\cong R[T]\otimes_R(P/(T)P)$. 
\end{Conjecture}

Important positive answers were given by 
 Quillen (see \cite{Q}) and Suslin (see  \cite{Su}) in dimension $\leq 1$ and Murthy \cite{Mu1} in dimension $2$. 
Later Lindel \cite{L} and Swan \cite{Mu} gave, in particular, positive answers for many regular local rings, essentially of finite type $\bf Z$-algebras.
 
 \begin{Theorem} (Lindel, Swan) \label{tls} Let $(R,\mm,k)$ be a regular local ring, essentially of finite type over $\bf Z$ and $p=$char $k$. The following statements hold.
 \begin{enumerate}
 \item If $p=0$ then $R$ is essentially smooth over its prime field.
\item If $p\not \in \mm^2$ then $R$ is essentially smooth over $\bf Z$.
\item If $p=0$ or  $p\not \in \mm^2$ then the BQ Conjecture holds for $R$.
\end{enumerate}
\end{Theorem}

 Then Swan noticed that it is  useful to have a positive answer to the following question.
 
 \begin{Question} (Swan \cite{Mu})\label{qs}  Is it a regular local ring a filtered inductive limit of regular local rings, essentially of finite type over $\bf Z$?
\end{Question} 
 
 The Bass-Quillen Conjecture (shortly BQ Conjecture) is connected with the following one.
 
 \begin{Conjecture} (Bass-Quillen-Suslin  Conjecture)  Let $R$ be a local ring, Assume either that  $R$ is regular local ring or that $1/r!\in R$. Let $v$ be a unimodular vector over $R[T]$ of length $(r+1)$. Then $v$ can be completed to an invertible matrix over $R[T]$.
 \end{Conjecture}
 This conjecture holds in many cases given for example by Rao \cite{Ra1}, \cite{Ra} who also proved the BQ Conjecture   for regular local rings of dimension $3$ with residue characteristic $\not = 2,3$. 
  
    For a regular local ring $(R,\mm,k)$ containing a field, or with $p:=$char $k\not \in \mm^2$, Swan's question has a  positive answer in \cite{P1}.  Using this partial positive answer and (iii) from Theorem \ref{tls} we got the following corollary (see \cite[Theorem 4.1]{P1} and also \cite[Theorems 2.1, 2.2]{S}).
    
 \begin{Corollary}  \label{c1}  
     The Bass-Quillen Conjecture holds for $R$ if $p=0$ or $p\not\in \mm^2$.
\end{Corollary}  
     
Recently, we gave a complete positive answer to Swan's Question   in \cite[Theorem 17]{P3} (see here Theorem \ref{tm}).

The purpose of this paper is to show that a positive answer to  the following question   gives a 
  complete positive answer to the BQ Conjecture (see Theorem \ref{bq}).
  
\begin{Question}  \label{q2} Let $(R,\mm)$ be a regular local ring, which is essentially smooth over ${\bf Z}_{(p)}$ and $b\in \mm^2$. Is it true the BQ Conjecture for the regular local ring $R/(p-b)$? 
\end{Question}
  We owe thanks to Ravi Rao for some useful comments.
\vskip 0.3 cm

\section{The Bass-Quillen Conjecture}
\vskip 0.3 cm

We start reminding  some definitions concerning smooth morphisms  after \cite{M}, or \cite{S}.
A ring morphism $R\to R'$ of Noetherian rings has  {\em regular fibers} if for all prime ideals $\pp\in \Spec R$ the ring $R'/\pp R'$ is a regular  ring.
It has {\em geometrically regular fibers}  if for all prime ideals $\pp\in \Spec R$ and all finite field extensions $K$ of the fraction field of $R/\pp$ the ring  $K\otimes_{R/\pp} R'/\pp R'$ is regular.
A flat morphism of Noetherian rings  is {\em regular} if its fibers are geometrically regular. If it is regular of finite type, or essentially of finite type  then it is called {\em smooth}, resp. {\em essentially smooth}.

The proof of our positive answer to Swan's Question says actually a little  more (see \cite[Theorem 17]{P3}).

\begin{Theorem} \label{tm} Every regular local ring $(R,\mm,k)$ with $0\not = p=$char $k\in \mm^2$  is a filtered inductive limit of regular local rings $R_i$,
 essentially smooth over a regular local $\bf Z$-algebra $A_i/(p-b_i)$, where $(A_i,\aa_i)$ is a regular local ring, essentially smooth over  
 ${\bf Z}_{(p)}$, and $b_i\in \aa_i^2$.
\end{Theorem}

In fact, this theorem gives also   the structure  of regular local rings essentially of finite type  over $\bf Z$ in the case $0\not=$char $k\in \mm^2$, which is not covered by (i),(ii) from Theorem \ref{tls}.

\begin{Corollary} \label{c2}
Let $B$ be a
 $\bf Z$-algebra regular local,  
  essentially of  finite type. Then $B$ has the form $B=A/(p-b) $, where $(A,\aa)$ is a regular local ring essentially smooth over $\bf Z$ and 
$b\in \aa^2$.
\end{Corollary}

\begin{proof} Applying Theorem \ref{tm} to $B$ we see that there exists a regular local ring $D$,
 essentially smooth over a regular local $\bf Z$-algebra $A/(p-b)$, where $(A,\aa)$ is a regular local ring, essentially smooth over  
 ${\bf Z}_{(p)}$, and $b\in \aa^2$ such that the identity of $B$ factors through $D$.   Then $B\cong D/q$ for some prime ideal $q\subset D$. Unfortunately, we cannot conclude that $B$ is among these $D$ (see the next remark).

 Note that  $D$ is a factor of an essentially smooth ${\bf Z}_{(p)}$-algebra $D'$ by $(p-b')$, where $b'$ is a lifting of $b$ to $D'$. Let $q'$ be the prime ideal of $D'$ containing $p-b'$ and such that $q'/(p-b')=q$. Changing $D$ by $D'$ and $q$ by $q'$ we may assume that $D$ is essentially smooth over   ${\bf Z}_{(p)}$. Then
$D$ is an   etale neighborhood of a localization of a polynomial algebra in $t$ variables $Y$ over  ${\bf Z}_{(p)}$ (see e.g. \cite[Theorem 2.5]{S}) and $(p-b,Y)$ generates the maximal ideal of $D$. Since $D$, $B$ are regular local we see that $q$ is generated by a part of a regular system of parameters of $D$, let us say $p-b,z$, where $z=(z_1,\ldots,z_r)$, $r\leq t$. After some linear transformations on $Y$ we may assume that $z_i=Y_i$, $1\leq i\leq r$. Then $B=D/q$ is an etale neighborhood of a localization of ${\bf Z}_{(p)}[Y_{r+1},\ldots,Y_t]/(p-b'')$, $b''$ being induced by $b$.
  \hfill\ \end{proof}

\begin{Remark}{\em Let $B$ be a regular local ring essentially of finite type over $\bf Z$, $D_n=B[X_n]$, $n\in {\bf N}$ and $\phi_{n,n+1}:D_n\to D_{n+1}$ be the $B$-morphism given by $X_n\mapsto 0$. Then $B$ is the limit of $(D_n,\phi_{n,n+1})$, the inclusion $B\subset D_n$ has a retraction and $B\not \cong D_n$ for any $n\in {\bf N}$.}
\end{Remark}

Next we will need the following two lemmas, the first one is elementary (see e. g. \cite[Theorem 4.2]{P1}).

\begin{Lemma} \label{l2} Let $R$ be a regular local ring, which is a filtered inductive limit of some regular local rings $(R_i)_{i\in I}$. If the BQ Conjecture holds for all $R_i$, $i\in I$, then it holds for $R$ too.
\end{Lemma}
\begin{proof} Let $M$ be a finitely generated projective  module over $R[T]$, $T=(T_1,\ldots,T_n)$. Then $M\cong R\otimes_{R_i} M_i$ for some   finitely generated projective  $R_i[T]$-module $M_i$. Indeed, if $M$ is defined by an idempotent $\phi$ from $\End(L)$ for some $L=R^t$ then we may find $i$ such that $\phi$ is extended from an endomorphism $\phi_i$ of $R_i^t$. Also we may find $i$ such that $\phi_i$ is idempotent, and  defines the wanted $M_i$. As BQ Conjecture holds for $R_i$ we get $M_i$ free and so $M$ is free too.
\hfill\ \end{proof}

The following lemma follows easily from \cite{L}. However, we give here a proof in sketch.

\begin{Lemma}\label{l1} Let $R\to R'$ be an essentially smooth morphism between regular local rings. If the BQ Conjecture holds for $R$ then it holds for $R'$ too.
\end{Lemma}
\begin{proof} $R'$ is an etale neighborhood of a localization of a polynomial algebra $A$ over $R$ (see e.g. \cite[Theorem 2.5]{S}). If the  BQ Conjecture holds for $R$ then it holds for $A$ too by \cite{R}. Now it is enough to apply the Corollary from \cite{L}.
\hfill\ \end{proof}

\begin{Theorem}\label{bq} If Question \ref{q2} has a positive answer then the BQ Conjecture holds for all regular rings.
\end{Theorem}

\begin{proof} By Quillen's Patching Theorem \cite[Theorem 1]{Q} we may prove the conjecture only for   regular local rings. Let $(R,\mm,k)$ be a regular local ring. Using \cite[Theorem 3.1]{P1} we may suppose that $0\not =p:=$char $ k \in\mm^2 $.

 By Theorem \ref{tm},  $R$  
 is a filtered inductive limit of some regular local rings  $D$,  essentially smooth over a regular local ring of the form $A/(p-b)$, where $(A,\aa)$ is a regular local ring, essentially smooth over  $\bf Z$ and $b\in \aa^2$.

  Then the BQ Conjecture 
 holds for $A$   by Lemma \ref{l1}  and for $A/(p-b)$ by Question \ref{q2}. Applying Lemma \ref{l1} it follows that BQ holds for $D$. The final result is a consequence of Lemma \ref{l2}. 
\hfill\ \end{proof}

We end the section with a special form of Theorem \ref{tm} in the frame of the discrete valuation rings (DVRs for short) in the idea of  \cite[Theorem 8]{P3}. Actually, \cite[Theorem 8]{P3} has a complicated proof given to illustrate \cite[Theorem 17]{P3} in the DVR case. The proof below is easier and does not use N\'eron Desingularization.

\begin{Theorem} Let $(A,\mm,k)$ be a DVR with $0\not = p=$char $k\in \mm^2$. Suppose that $k$ is separably generated over ${\bf F}_p$. Then $A$  is a filtered inductive union of DVRs
essentially  of finite type over $\bf Z$.
\end{Theorem}

\begin{proof} As in \cite[Theorem 8]{P3}, let $y=(y_i)_{i\in I}$ be a system of elements of $A$ inducing a separable transcendence base $(\bar y)$ of $k$ over ${\bf F}_p$. Then $C_0={\bf Z}[Y]_{p{\bf Z}[Y]}$ for some variables $Y=(Y_i)_{i\in I}$ is a DVR and the map $C_0\to A$,
$Y\to y$ defines a ramified extension inducing an algebraic separable residue field extension ${\bf F}_p( {\bar y})\subset k$.

Note that $A$ is a filtered inductive union of DVRs $A_L=A\cap L$ with $L\subset $Fr$(A)$ a finite type field extension of ${\bf Q}(y)$. We claim that $L$ must be finite over ${\bf Q}(y)$. Indeed,
assume that $z\in L $ is transcendental over ${\bf Q}(y)$ and $\mm^s\subset A$ for some $s\in {\bf N}$. Choose $r$ such that $p^r>s$ and consider the DVR extension $A_{L''}\subset A_{L'}$ for $L'={\bf Q}(y,z)$ and $L''={\bf Q}(y,z^{p^r})$. The residue field extension induced by $A_{L''}\subset A_{L'}$ is pure inseparable and also separable by assumption. Then it is trivial and so the ramification index of   $A_{L''}\subset A_{L'}$ is $p^r>s$. Contradiction!

Since $L$ is finite over ${\bf Q}(y)$ it follows that $A_L$ is a localization of the integral closure of $C_0$ in $L$ and so essentially finite over $C_0$, which is enough.
\hfill\ \end{proof}

\section{Swan's question in the non-reduced case}

We start this section reminding the first part of \cite[Theorem 17]{P3}.

\begin{Theorem} \label{s}  Let $(A,\mm,k)$ be a Noetherian local ring, $s=(s_1,\ldots,s_m)$ some positive integers and $\gamma=(\gamma_1,\ldots,\gamma_m)$ a system of nilpotents of $A$. Suppose that $0\not =p\in \mm^2$, $R=A/(\gamma)$ is a regular local ring and  $A$ is a flat  $\Phi={\bf Z}_{(p)}[\Gamma]/(\Gamma^s)$-algebra, $\Gamma\mapsto \gamma$ with $\Gamma=(\Gamma_1,\ldots,\Gamma_m)$ some  variables, and $(\Gamma^s)$ denotes the ideal $(\Gamma_1^{s_1}\cdots \Gamma_m^{s_m})$.
 Then $A$ is a   filtered inductive limit of some  Noetherian local $\Phi$-algebras $(F_i)_i$ essentially of finite type with $F_i/\Gamma F_i$ regular local rings. 
\end{Theorem}

\begin{Remark} {\em  The above theorem holds also when $p=0$ or $p\not\in \mm^2$ as in \cite[Theorem 4.1]{P1} because the map $\Phi\to A$ is regular in this case.}
\end{Remark}

\begin{Lemma} \label{l4} In the notation and hypothesis of the above theorem, let $(B,\bb)$ be a Noetherian local ring,  and $z=(z_1,\ldots,z_m)$ a system of elements of $B$. Suppose that
\begin{enumerate}
\item for all $i\in [m]$ $s_i>1$, $z_i^{s_i}=0$,
\item for   all $i\in [m]$ and $j\in [s_i-1]$, $(z_1,\ldots,z_{i-1}):z_i^j= (z_1,\ldots,z_{i-1},z_i^{s_i-j})$,
\item $B/(z)$ is a regular local ring.
\end{enumerate}
Then $B$ is a flat $\Phi$-algebra by the map $\phi:\Phi \to B$, $\Gamma \mapsto z$.
\end{Lemma}
\begin{proof} Suppose $m=1$. Note that the minimal free resolution of $\Phi/(\Gamma)$,
$$ ...\to \Phi/(\Gamma^s)\xrightarrow{\Gamma}\Phi/(\Gamma^s)\xrightarrow{\Gamma^{s-1}}\Phi/(\Gamma^s)\xrightarrow{\Gamma}\Phi/(\Gamma^s)\to \Phi/(\Gamma)\to 0$$
gives after tensorizing with $B$ a minimal free resolution of $B/(z)$. Thus \\
$\Tor_1^{\Phi}(\Phi/(\Gamma),B)=0$. As $B/(z)$ is a flat $\bf Z$-algebra we get $\phi$ flat using the Local Criteria of Flatness \cite[Theorem 49, (20.C)]{M}. 

Induct on $m$ and assume $m>1$. Then $B/(z_m)$ is a flat  $\Phi/(\Gamma_m) $-algebra by induction hypothesis. As in the case $m=1$  we see that 
$\Tor_1^{\Phi}(\Phi/(\Gamma_m),B)=0$ and so $\phi$ is  flat using the Local Criteria of Flatness.
\hfill\ \end{proof}

\begin{Remark} {\em Let $R$ be a regular local ring, $s=(s_1,\ldots, s_m)\in {\bf N}^m$
and $z=(z_1,\ldots,z_m)$ be a part of a regular system of parameters of $R$. Then\\
 $B=R/(z_1^{s_1},\ldots,z_m^{s_m})$ satisfies the above lemma.}
\end{Remark}

Using the above lemma and Theorem \ref{s} we get the following result.

\begin{Theorem} \label{t4}  In the notation and hypothesis of the above lemma, $B$ is a   filtered inductive limit of some  Noetherian local $\bf Z$-algebras $(F_i)_i$ essentially of finite type with $F_i/\sqrt{(0)}$ regular local rings. 
\end{Theorem}

As in Corollary \ref{c2} we get the following result.

\begin{Corollary} \label{c3} In the notation and hypothesis of Lemma \ref{l4}, suppose that $B/(z)$ is essentially of finite type over $\bf Z$. Then $B$ is  
 essentially smooth over a local $\Phi$-algebra $D$    of type $C/(p-b)$, where $(C,\qq)$ is a  local ring, esentially smooth over $\Phi$ with $b\in \qq^2$ (so $D/ \Gamma D$ and $C/\Gamma C$ are regular).
\end{Corollary}

\begin{Theorem} \label{bqn}  Let $(A,\mm,k)$ be a Noetherian local ring,  $T=(T_1,\ldots,T_n)$ some variables and $\bb$ a nilpotent ideal of $A$. Suppose that 
 $R=A/\bb$ is a regular local ring and   the BQ Conjecture holds for $R$ (for example $p:=\mbox{char}\ k=0$, or $p\not \in \mm^2$).
Then any  finitely generated projective $A[T]$-module is free. 
\end{Theorem}

\begin{proof}
Let $M$ be a  finitely generated projective $A[T]$-module. Then ${\bar M}=M/\bb M$ is finitely generated projective over $R[T]$ and so it is free. Let $x=(x_1,\ldots,x_r)$ be a system of elements from $M$ inducing a basis in $\bar M$. Then $M=<x>+\bb M=\\
<x>+\bb^tM=<x>$ for $t\in {\bf N}$ with $\bb^t=0$. Thus the map $\phi:A[T]^n\to M$ given by $(a_1,\ldots,a_n)\mapsto \sum_ia_ix_i$ is surjective. Set $N=\Ker  \phi$. Tensorizing with $A[T]/(\bb)$ the exact sequence
$$0\to N\to A[T]^n\to M\to 0$$
we get the following exact sequence
$$\Tor_1^{A[T]}(M,A[T]/(\bb))\to N/\bb N\to R[T]^n\to {\bar M}\to 0,$$
where the first module is zero because $M$ is a flat $A[T]$-module. As the last map is injective we see that $N=
\bb N=\bb^tN=0$, that is $\phi$ is an isomorphism.
\hfill\ \end{proof}

\end{document}